\documentclass[10pt]{amsart}

\usepackage{amsmath,amscd,amsfonts,amsthm,amssymb,kpfonts,times,comment, hyperref}

\DeclareMathOperator{\SL}{SL}
\DeclareMathOperator{\PGL}{PGL}
\DeclareMathOperator{\GL}{GL}

\newcommand{\et}{\mathrm{et}}
\DeclareMathOperator{\Spec}{Spec}

\DeclareMathOperator{\Gal}{Gal}

\newcommand{\field}[1]{\mathbb{#1}}
\newcommand{\disc}{\mathrm{disc}}

\newcommand{\Q}{\field{Q}}

\newcommand{\Z}{\field{Z}}

\newcommand{\R}{\field{R}}
\newcommand{\C}{\field{C}}
\renewcommand{\P}{\field{P}}

\newcommand{\ra}{\to}

\newcommand{\LL}{\mathcal{L}}

\newcommand{\eps}{\epsilon}

\newcommand{\bs}{\backslash}

\newcommand{\beq}{\begin{displaymath}}
\newcommand{\eeq}{\end{displaymath}}
\newcommand{\beqn}{\begin{equation}}
\newcommand{\eeqn}{\end{equation}}

\newcommand{\dR}{\text{dR}}
\newcommand{\an}{{\text{an}}}

\theoremstyle{plain}
\newtheorem{thm}{Theorem}[section]

\newtheorem{cor}[thm]{Corollary}
\newtheorem{lem}[thm]{Lemma}

\theoremstyle{definition}
\newtheorem{defn}[thm]{Definition}

\newtheorem{exmp}[thm]{Example}
\newtheorem{rem}[thm]{Remark}

\theoremstyle{remark}

\begin{document}

\title{Sparsity of integral points on moduli spaces of varieties}

\begin{abstract}
Let $X$ be a quasi-projective variety over a number field, 
admitting (after passage to $\mathbb{C}$) a geometric variation of Hodge structure whose period mapping
has zero-dimensional fibers. 
Then the integral points of $X$ are {\em sparse}:  the number of such points of height $\leq B$ grows slower than any positive power of $B$. 
 
For example,  homogeneous integral polynomials in a fixed number of variables and degree, with discriminant  divisible only by a fixed set of primes,  are  sparse
 when considered up to integral linear substitutions.

\end{abstract}

\author{Jordan S. Ellenberg, Brian Lawrence,  and Akshay Venkatesh} 

\maketitle

\section{Introduction}

Let $K \subset \C$ be a finite field extension of the rational numbers,
and $S$ a finite set of primes of $K$.  
We will  consider  $S$-integral points on quasi-projective $K$-varieties $X^{\circ} \subset \mathbb{P}^m$.

More precisely: 
in this situation we  will write $X$  for the Zariski closure of $X^{\circ}$ in $\mathbb{P}^m$,
 $\LL = \mathcal{O}(1)$  the associated hyperplane bundle, and   $Z$ for $X \backslash X^{\circ}$.
After choosing a good integral model of $X$ and $Z$ (see \S \ref{integralmodel} for details)
we obtain a notion of ``$S$-integral point of $X^{\circ}$,''
and the projective embedding allows us to refer to the ``height'' of such a point.   
(By ``height'' we always mean the \emph{multiplicative} height.)
We will denote by $X_{\C}$
the complex variety obtained by base extension via the fixed embedding $K \hookrightarrow \C$,
by $X(\C)$ its complex points,   
and by $X^{\an}$ its complex analytification.

\begin{thm} \label{thm1}
Let $X^{\circ} \subset \mathbb{P}^m$ be a quasi-projective variety over $K$
such that $X^{\circ}_{\C}$ admits a geometric variation of Hodge structure, 
whose associated period map (see \cite{Griffiths} for definitions)  is locally finite-to-one, i.e. has zero-dimensional fibers.  Then integral points on $X^{\circ}$ are sparse, in the sense that
\begin{equation}
\label{th:main}\# \{x \text{ an $S$-integral point of $X^{\circ}$, of (multiplicative) height at most $B$} \} =O_{\eps}(B^{\eps}).  \end{equation}
  \end{thm}
  By ``geometric variation of Hodge structure''
we mean a direct summand of a VHS arising from a smooth projective family over $X^{\circ}_{\C}$; in particular the Hodge structures arising here are pure.  
In the situation of the Theorem, the ``period map'' is understood to be the complex-analytic map $\widetilde{X^{\circ \an}} \rightarrow D$
classifying the VHS, where $\widetilde{X^{\circ \an}}$ is the universal cover of $X^{ \circ \an}$ and $D$
is the period domain associated to the varying  Hodge structures. 

We note that the condition we are imposing on $X^{\circ}$ depends only on the complex algebraic variety $X^{\circ}_{\C}$ and not on its rational form over $K$. 

 The notation $O_{\eps}(B^{\eps})$ is that of analytic number theory:
for each  $\eps > 0$ there exists $c_{\eps}$ such that
the left hand side is at most $c_{\eps} B^{\eps}$.  
 
Theorem \ref{thm1} will be deduced from the following more general theorem.

\begin{thm}  \label{thm2}
Let $\pi: \mathfrak{X} \rightarrow X^{\circ}$
be a projective smooth morphism of $K$-varieties,  
and (for some $i \geq 0$) let $\Phi$ be the period map   
associated to the variation of Hodge structure
 $\mathsf{V} = R^i \pi_* \Q$ on $(X^{\circ})^{\an}$. 
 
 Then 
 the $S$-integral points of $X^{\circ}$ with height at most $B$
 are covered by $O_{\eps}(B^{\eps})$ geometrically irreducible
 $K$-varieties, each lying in a single fiber of $\Phi$.
 
\end{thm}

We have abused language in the statement: for an irreducible analytic subvariety $V \subset X^{\circ \an}$ we shall  say that $V$ ``lies in a fiber of $\Phi$'' if some component (hence every component)
of the preimage of $V$ in the universal cover
lies in such a fiber.  We will use similar language elsewhere in the paper -- if a property of the period map is local,
we will phrase it on $X^{\circ}$ rather than the universal cover.

The bound \eqref{th:main} applies to many natural moduli spaces of varieties.   
For instance, one may take $X$ to be the projective space parametrizing hypersurfaces 
of a given degree and dimension, 
and $Z$ the locus of singular hypersurfaces, and obtain the following corollary, which we will derive from Theorem~\ref{thm1} in \S \ref{Corproof}.
\begin{cor} \label{Corhyp} Fix $n \geq 2$ and $d \geq 3$ and a
finite set $S$ of rational primes.  

Then the   $S$-integral homogeneous degree $d$ polynomials $P= \sum a_{i_1\dots i_n} x_1^{i_1} \dots x_n^{i_n}$ in $n$ variables with $S$-integer discriminant  and 
$\max_{i, v \in S} |a_i|_v \leq B$  
lie in at most $O_{S,d,n,\eps}(B^{\eps})$ orbits of  the group of  integral linear substitutions $\GL_n(\Z[\frac{1}{S}])$.

Similarly the set of such integral polynomials (i.e. with $\Z$ coefficients), 
with discriminant exactly equal to a nonzero
integer $N \in \mathbb{Z}$,  and $\max|a_i| \leq B$,
lie in at most $O_{N,d,n,\eps}(B^{\eps})$ orbits of  the group of unimodular  integral linear substitutions $\SL_n(\Z)$. \end{cor}

 Recall that the {\em discriminant} of a homogeneous polynomial is a certain homogeneous polynomial in the coefficients of $P$
which vanishes precisely when the hypersurface defined by $P=0$ is singular; see \cite{PS} for more discussion of its significance.

 In the paper \cite{LV18} of the second- and third- named author, it is proved by a much more complicated argument that
for ``large enough'' $n,d$ the  
set of polynomials with $S$-integer discriminant  is not Zariski dense in the ambient affine hypersurface (which of course neither implies nor is implied by what we prove here).
 It is expected (\cite[Conj.\ 1.4]{JL}) that this set is {\em finite} as part of a ``Shafarevich-type'' statement
about hypersurfaces with good reduction;
in fact, if one assumes the Lang-Vojta conjecture, 
then results on the hyperbolicity of period domains
imply finiteness of $S$-integral points in the general context of Theorem \ref{thm2}; see \cite{Zuo} and \cite[Thm.\ 1.5]{JL}.\footnote{
The theorem in loc.\ cit.\ is stated only for complete intersections, 
but the proof applies in the generality of our Theorem \ref{thm2}.}
We note that the Shafarevich assertion can hold even in cases where there is no Torelli theorem; for instance, it holds for the case $(n,d) = (4,3)$ of cubic surfaces, by a result of Scholl~\cite{scholl:delpezzo}. 
In addition to cubic surfaces, the Shafarevich assertion is also known for cubic and quartic threefolds \cite{JL}.

 \subsection{Discussion of the proof}
 
 The reduction of Theorem \ref{thm1} to Theorem \ref{thm2}   is relatively formal, and so we will focus
 on the latter Theorem here.

 The idea of the proof goes back to ideas initiated in the work of Bombieri-Pila \cite{BP}, Heath-Brown \cite{HBR}
 (and, independently and in a different context, Coppersmith \cite{Coppersmith}): 
 One can show that an excess of rational (or integral) points on any $X$ as above must lie on the zero-loci
 of some auxiliary functions, i.e., must be covered by a certain collection of divisors.
Then we iterate the process, replacing $X$ by these divisors, and covering the points by codimension-$2$ subvarieties, etc.; since we have almost no control over what these subvarieties look like, it is crucial to have, as in Bombieri-Pila and Heath-Brown, that the basic bounds are {\em uniform} in the ambient variety. See Remark \ref{Brobremark}
for a quick review of this method.

 Unfortunately, this process is quite lossy when executed in dimensions larger than $1$.
 The problem is the usual one in studying rational points on higher dimensional varieties: Even if a variety
 has high degree, it may contain subvarieties of much lower degree, such as hyperplanes. 
For this reason, there  are very few situations where one can achieve $B^{\eps}$-type bounds (see e.g.\ \cite{PW} for an example 
in a different context). 

The key point of our strategy here is to use the fact that the bounds of \cite{Broberg} improve
under {\'e}tale covers, which are plentiful under the conditions on $\mathfrak{X} \rightarrow X$ we have imposed; these covers can be used to raise the degree of intermediate subvarieties.   To deploy this in a way suitable for an iterative argument, we use 
the global invariant cycle theorem of Hodge theory to construct a  cover $\widetilde{X} \rightarrow X$
 which remains nontrivial when restricted to any subvariety of $X$.  In some sense, this strategy generalizes the approach of \cite{EV:2d}, which used \'etale covers to slightly improve on Heath-Brown's bounds in case $X$ is a non-rational curve (i.e. a curve with interesting \'etale covers); in the case $\dim X = 1$, of course, the issue of restriction of covers to subvarieties becomes trivial.

In the end, the core part of the proof is quite short -- the main idea is the induction on dimension laid out in  \S \ref{conclude}, with the key induction step being provided by Lemma \ref{lem:induction}. The reader may want to 
skip directly to these, and refer back as necessary.   One of the reasons for the length of the text is   
 the technicalities
involved in making various results uniform over all subvarieties of $X$.

One way of thinking about the strategy executed here is in terms of ``profinite repulsion" between low-height integral points.  For simplicity of exposition we take $K=\Q$ for this paragraph.  For any prime $p$, the embedding of $X(\Z[1/S])$ into $X(\Q_p)$ induces a topology on the former set, and the methods of Bombieri-Pila and Heath-Brown rely on an argument that low-height points tend to repel each other in this topology.  In the cases considered in this paper, there is an alternate profinite topology on $X^{\circ}(\Z[1/S])$ afforded by the map
 \beq
 X^{\circ}(\Z[1/S]) \ra H^1(G_\Q, \pi_1(X^{\circ}_{\overline{\Q}}))
 \eeq
 and our approach can be thought of as exploiting the fact that in this topology, too, the low-height $S$-integral points repel each other.  The two profinite repulsion phenomena working in tandem is what allows us to get a better upper bound on the number of low-height points.  To be more precise:  to say that two points $P_1$ and $P_2$ of $X^{\circ}(\Z[1/S])$ are close together in the alternate profinite topology is to say that there is some high-degree cover $\widetilde{X}^\circ \ra X^\circ$, with $\widetilde{X}^\circ$ irreducible, such that $P_1$ and $P_2$ both lie in the image of $\widetilde{X}^\circ(\Q)$.  Then the method of Heath-Brown is used to control the low-height points of $\widetilde{X}^\circ(\Q)$, which is aided by the fact that this new variety has very high degree.

In the course of writing this paper we learned of the paper \cite{Br} of Brunebarbe, which uses the same  basic idea of increasing the positivity of line bundles by passing to {\'e}tale covers, in the context of a VHS with finite period map, in order to prove uniform algebro-geometric assertions about covers of varieties carrying such a VHS.  Indeed
Brunebarbe gives a more analytic (and therefore effective) proof of a result (\cite[Theorem 1.13]{Br}) of the type of Lemma \ref{lem31}, whereas
our argument uses general finiteness statements of algebraic geometry. As in our paper, a crucial input to \cite{Br} 
 is the invariant cycle theorem of Hodge theory.

Finally, let us mention some potential refinements of the result.
\begin{itemize}
\item The reader will observe that the {\em geometricity} of the variation of Hodge structure is,
in fact, barely used at all in our proof. It permits us to phrase
the argument in quite an explicit way, but it should
be possible to entirely avoid it.  
\item Indeed, the whole apparatus of  variations of Hodge structures is used simply to provide examples of varieties $X^\circ$ with the property that the image of $\pi_1(W^{\circ \an}) \ra \pi_1(X^{\circ \an})$ has infinite order for every subvariety $W^\circ$ of $X^\circ$.  We note that this is essentially the same as Koll\'{a}r's condition of {\em large fundamental group} from \cite{kollar}. 
For an $X^{\circ}$ which possessed this property for some reason unrelated to Hodge structures, the arguments here should work just as well to control its rational points.  For example, abelian varieties do not admit variations of Hodge structure but they do have large fundamental group, so the arguments of this paper should in principle imply sparseness of rational points; but for abelian varieties, we know this already, by the Mordell-Weil theorem.
\item One might ask what happens for varieties $X^\circ$ carrying a variation of {\em mixed} Hodge structures; it seems plausible that a theorem like the one proved here will still hold. 
 \item It would be interesting to refine our statement $O_{\eps}(B^{\eps})$ to a more precise upper bound.  To do this would, at the very least, require that the inexplicit dependence on degree in Broberg's bounds be made explicit.  Questions in this vein have been the topic of much recent progress; see, e.g., the recent result of Castryck, Cluckers, Dittman, and Nguyen~\cite[Theorem 2]{CCDN}, which gives bounds with explicit polynomial dependence on degree.
The quantitative approach of \cite{Br} will also likely be needed in order to get effective lower bounds for the degrees of covers that arise in the argument. 
\end{itemize}

 \subsection{Notation and integral models} \label{integralmodel}
 
 As above, $K$ is a number field embedded in $\C$;  
 let $\overline{K}$ be the algebraic closure of $K$ in $\C$ (so $\overline{K}$ is just another name for $\overline{\mathbb{Q}}$). 
 
 By a variety $V$ (over $K$), or $K$-variety for short,  we mean, as usual, an integral, separated, finite-type scheme over $K$.
 In this situation, 
 $V$ will always denote the scheme over $K$, 
 and $V_{\C}$ the base extension of $V$ to a complex scheme via $K \hookrightarrow \C$. 
 (Note that a $K$-variety is not required to be geometrically irreducible; thus, $V_{\C}$ may have multiple components.) 
We will sometimes use notation such as $Y_{\C}$ even when $Y$ is not a base-change from $K$,
simply to emphasize that $Y_{\C}$ is a complex variety.  
 We will write $V^{\an}$ for the (complex) analytification of $V_{\C}$.

As above let $S$ be a finite set of primes of $K$. 
 Let $\mathfrak{o}_S$ be the ring of $S$-integers of $K$,
 i.e.\ the set of those elements of $K$ that are integral outside $S$.

\begin{defn} \label{goodmodel}
Let $X^{\circ} \subset \mathbb{P}^m$  
be a locally closed subvariety, and $X$
its Zariski closure. 
Let $Z$ be the complement $X \backslash X^{\circ}$, a Zariski-closed subset of $\mathbb{P}^m$. 
 
 We will call a \emph{good integral model} for $(X, Z)$ a choice of a
projective flat $\mathfrak{o}_S$-scheme $X_S \subset \mathbb{P}^N_S$ 
extending $X$ on the generic fiber. 
Given a good integral model, we define $Z_S$ to be the Zariski closure of $Z$ inside $X_S$, and take $X^{\circ}_S = X_S - Z_S$.

If we are given a projective smooth morphism $f: \mathfrak{X} \rightarrow X^{\circ}$  
and we have chosen a good integral model for $(X,Z)$, a good integral model for $f$ will be
a projective smooth morphism
$f_S:  \mathfrak{X}_{S} \rightarrow X^{\circ}_S$. 
\end{defn}

Good integral models exist after possibly increasing $S$, by standard spreading arguments 
(a nice reference here is Theorem 3.2.1 and Appendix C, Table 1 of \cite{PoonenRP}).

A $K$-point of $X$ will be said to be {\em $S$-integral}
if it extends to a morphism $\mathrm{Spec} (\mathfrak{o}_S) \rightarrow X_S-Z_S$. 
Explicitly,  assuming $\mathfrak{o}_S$ to have class number one for simplicity, 
any $K$-point, represented by a point
 with relatively prime homogeneous coordinates $\mathbf{x} = [x_0: \dots: x_N] \in \mathfrak{o}_S^{N+1}$,
is \emph{integral} at a prime $p \notin S$ if there exists
  a homogenous function $f \in \mathfrak{o}_S[x_0, \dots, x_N]$
 in the ideal of $Z_S$ such that   $f(\mathbf{x})$ 
is nonzero mod $p$; it is $S$-integral if it is integral at all primes $p \notin S$. 

 If we choose two different choices of good integral model as above, there exists a finite set $T$
such that any $S$-integral point for the first model is $S \coprod T$-integral
for the second model, so the choice of good integral model is irrelevant to the statement of the theorem.

 Finally, we fix some other notation to be used:
For any field $F$ and any
$F$-variety $Y$ mapping to $X$,  
we denote by $Y^{\circ}$ the preimage of $X^{\circ}$ inside $Y$. 
If the map $f:Y \rightarrow X$ is finite, the ``degree'' of $Y$ will  
 be the degree of the pullback $f^* \LL$ of the hyperplane bundle  $\LL$, characterized by the asymptotic formula
\begin{equation} \label{growth}  \dim_F \Gamma(Y, f^* \LL^{\otimes k}) \sim \frac{ \mathrm{deg} Y}{(\dim Y)!} k^{\dim Y} \end{equation} 
for large $k$.

\subsection{Acknowledgements}
We are grateful to Michael Kemeny for helpful discussions about the quasi-finiteness of period maps,
and to Brian Conrad for help and corrections to \S \ref{res}.   The first author was supported by NSF grant DMS-2001200,
the second by NSF grant DMS-2101985,
and the third by NSF grant DMS-1931087.

\section{Some results used in the proof}
  \label{res}
We collect results we will use in the proof. 
These
are a few lemmas from algebraic geometry, and a crucial
bound on rational points due to Broberg.

We emphasize again that, for $F$ a field, a variety over $F$  
is assumed reduced and irreducible, but not geometrically irreducible. 

\begin{lem} \label{KS}
Suppose that $\pi: \mathfrak{X}  \rightarrow X$
is a projective smooth morphism of   algebraic varieties
over a subfield $\kappa \subset \C$, with
$X$ smooth, 
and let $\mathcal{T}$ be the tangent bundle of $X$. 
Let
$\Phi$ be the period map   
associated to the variation of Hodge structure 
 $\mathsf{V} = R^i \pi_* \Q$ on $(X)^{\an}$,
for some $i>0$. 

Then there is a morphism 
\begin{equation}
\label{gdef} g \colon \mathcal{H}_1 \otimes \mathcal{T} \rightarrow \mathcal{H}_2\end{equation}
of vector bundles
over $X$ (everything defined over $\kappa$)  such that
the derivative of the period map along a tangent vector $t$ at $x$ vanishes if and only if 
 $g_x( -, t)$  
is the zero map from $\mathcal{H}_1$ to $\mathcal{H}_2$.

\end{lem}
\proof

This is a consequence of standard facts in Hodge theory, in particular the algebraicity of the Gauss--Manin connection (see \cite{Katz_Oda} and \cite[\S 1]{Katz_nilp}).

The $i$th relative algebraic de Rham cohomology of $\mathfrak{X}/X$ gives a variation of Hodge structure $H_{\dR}$ on $X$.
As a variation of Hodge structure, $H_{\dR}$ is a vector bundle, 
equipped with a flat connection and a filtration by algebraic subbundles $F^p H_{\dR}$.
These data determine the period map $\widetilde{X^{\circ \an}} \rightarrow D$,
where $D$ is a flag variety classifying filtrations 
of a fixed vector space of dimension $\operatorname{rank} H_{\dR}$
by subspaces of dimensions $\operatorname{rank} F^p H_{\dR}$.

For each step $F^p H_{\dR}$ of the filtration, the connection on $H_{\dR}$ defines an $\mathcal{O}_X$-linear map
\[ g_p \colon F^p H_{\dR} \otimes \mathcal{T}  \rightarrow H_{\dR} / F^p H_{\dR}, \]
which describes how the subspace $F^p H_{\dR, x}$ varies in $x \in X$.
Combining the maps $g_p$ over all $p$, we obtain
\[ g \colon \bigoplus_p F^p H_{\dR} \otimes \mathcal{T} \rightarrow  \bigoplus_p  H_{\dR} / F^p H_{\dR}.  \]
This map $g$ determines the differential of the period map.
More precisely, at every point, the tangent space to $D$ is identified with a subspace of
\[ \operatorname{Hom} \left (  \bigoplus_p F^p H_{\dR}, \bigoplus_p  H_{\dR} / F^p H_{\dR}  \right ), \]
and with this identification, $g$ gives the differential of the period map.

Finally, we note that $g$ is defined over $\kappa$ by  \cite{Katz_Oda} and \cite[\S 1]{Katz_nilp}.

\qed

\begin{lem} \label{degree}
Let $F$ be a field of characteristic zero.

Suppose that $Y$ is a proper $F$-variety equipped with an ample line bundle $\LL$,
and let $g: \widetilde{Y} \rightarrow Y$ be finite, with $\dim(\widetilde{Y}) = \dim(Y)$. 
Writing $\deg g$ for the degree of $g$ at the generic point, we have 
 \begin{equation} \label{degfinite} \deg_{g^* \LL}(\widetilde{Y}) = (\deg g) \deg_{\LL}(Y).\end{equation}
 \end{lem}
 
 We will apply this only when $\widetilde{Y}$ is a variety, but the argument does not require that, taking
 \eqref{growth} as the definition of degree.

\proof
By the projection formula we have
   $$\Gamma(\widetilde{Y}, g^* \LL^{\otimes k}) = \Gamma(Y, g_* \mathcal{O} \otimes \LL^{\otimes k}).$$
Now $g_* \mathcal{O}$ is isomorphic to $\mathcal{O}^{\oplus (\mathrm{deg} g)}$
away from a set of positive codimension on $Y$. Therefore
 $\Gamma(\widetilde{Y}, g^* \LL^{\otimes k})$
 coincides with 
 $(\mathrm{deg} \ g) \frac{ \mathrm{deg}_{\LL} Y}{(\dim Y)!} k^{\dim Y}$
 for large $k$, up to terms $O(k^{\dim Y-1})$.  We conclude using $\dim \widetilde{Y} = \dim Y$. 
\qed 

The next Lemma 
 is  \cite[Expos{\'e} I, Corollaire 10.8]{SGA1}.

\begin{lem} \label{disjointness}
Suppose that $X$ is an irreducible normal noetherian scheme
and $f: Y \rightarrow X$ a finite {\'e}tale cover. Then the irreducible
components of $Y$ are disjoint.
\end{lem}

The following result asserts, essentially, ``boundedness''
of the set of irreducible varieties in a fixed projective space and bounded degree.  
Results of this type are also stated and used in work of Salberger \cite[Lemma 1.4, Thm.\ 3.2]{Salberger}
in a similar context; in the interest of self-containedness 
we give a proof of precisely what we use.

\begin{lem} \label{Kleiman}
Let $F$ be a field of characteristic zero. 

\begin{itemize}
\item[(a)] Suppose that $V \subset \mathbb{P}^m_F$ is a closed subvariety   (irreducible, reduced closed subscheme) of degree $d$.  
Then there are bounds, depending only on $m,d$ (not on $F$) for each coefficient of the Hilbert polynomial of $V$,
and in particular the homogeneous ideal $I(V)$  of $V$ is generated in degree $O_{m,d}(1)$. 

\item[(b)] Suppose given integers $m, n, d, R$.  Then there exist bounds $D$ and $N$ with the following property: 
Suppose $V_1, \ldots, V_r$ (with $r \leq R$) is a collection of closed subvarieties of $\mathbb{P}^m_F$, 
with each $V_i$ of dimension $\leq n$ and degree $\leq d$. 
 Let $Z$ be the intersection
\begin{equation} \label{Zdef} Z = \bigcap_{i=1}^{r} V_i. \end{equation}
Then the number of irreducible components of $Z$ is at most $N$,  
and the degree of each such component (endowed with the reduced scheme structure) is at most $D$.  
Furthermore, the bounds $D$ and $N$ are independent of the field $F$.

 \item[(c)] Suppose that $V \subset \mathbb{P}^m_F$ is a  
closed  variety of dimension $n$ and degree $d$ which is {\em not} geometrically irreducible. 
 Then there exist finitely many subvarieties $V_1, \ldots, V_N \subseteq V$, defined over $F$, such that:
\begin{itemize}
\item Each $V_i$ is irreducible of dimension $\leq n-1$,
\item $V(F) = \bigcup V_i(F)$, and
\item the number $N$ of $V_i$'s, and degree of each $V_i$ can be bounded in terms of $n, m, d$ (but independently of the field $F$ and the variety $V$).  
\end{itemize}

\item[(d)]  Suppose that $V \subset \mathbb{P}^m_F$ is a closed variety of dimension $n$ and degree $d$.  Then   
the set of points in $V(F)$ which are singular on $V$ can again be covered by varieties $V_1, \dots, V_N$
with the same properties as (c). 
 \end{itemize}

\end{lem}

\begin{exmp}
 As an example of ``bad'' examples for part (a): take 
 a $d$-dimensional variety, and adjoin to it a large set of disjoint points;
this modification does not affect the degree, and shows the need for irreducibility or at least
 equidimensionality. Similarly, consideration of
 embedded points shows that ``reduced'' is also important. 

As an example of the situation in part (c), consider the plane curve defined by $x^2 + y^2 = 0$.
This is irreducible over $\mathbb{Q}$ but not geometrically irreducible;
it only has one rational point $(0, 0)$, which is contained in a zero-dimensional, geometrically irreducible subvariety.
More generally, suppose $F$ is a number field, and consider ``the affine line over $F$, with the origin reduced to a $\mathbb{Q}$-point'' -- that is, $\operatorname{Spec} ( \mathbb{Q} + T F[T])$.
If $F \neq \mathbb{Q}$, this scheme is irreducible, but geometrically reducible, and its only $\mathbb{Q}$-point is the origin.
(Taking $F = \mathbb{Q}[i]$ recovers the original example.)
 \end{exmp}

\proof

 The first assertion of (a) follows from \cite[Expos\'e XIII, Corollary 6.11(a)]{SGA6}.
 That assertion applies as formulated to ``special positive cycles'' over an algebraically closed field;
 in our situation $V_{\overline{F}} \subset \mathbb{P}^m_{\overline{F}}$
 is   reduced equidimensional and its decomposition into irreducible components give the closed
 subschemes appearing in {\em loc. cit.} D{\'e}finition 6.9.
 
The consequence 
on bounded generation of $I(V)$ follows, because such a bound on generation of
the defining ideal is valid in any finite type subscheme
of the Hilbert scheme.

For (b), we may as well consider one fixed $r$.
Let $\mathcal{P}$ be the finite set of polynomials arising from (a),
and  let $\operatorname{Hilb}$ be the Hilbert scheme parametrizing closed subschemes of $\mathbb{P}^m$
with Hilbert polynomial in $\mathcal{P}$.  This is a finite-type $\mathbb{Q}$-scheme by (a). 
Now tuples $(V_1, \ldots, V_r)$ are classified by suitable $K$-points of $\operatorname{Hilb}^r$, which is again of finite type; 
and the result now follows from standard results on families over a finite-type base.

Specifically, we have the universal schemes $\mathcal{V}_1, \ldots, \mathcal{V}_r$ over $\operatorname{Hilb}^r$; 
let $\mathcal{Z} \subseteq \mathbb{P}^m \times \operatorname{Hilb}^r$ be their fiber product over $\mathbb{P}^m$,
so fiberwise $\mathcal{Z}$ gives the intersection of the $\mathcal{V}$s (although
with a possibly non-reduced schematic structure). We will work by Noetherian induction.
Let $\eta$ be the generic point of a closed irreducible subscheme $H \subseteq \operatorname{Hilb}^r$.
We will show that there exists a relatively open subset $U \subseteq H$ (that is, open in $H$, but not necessarily in $\operatorname{Hilb}^r$)
such that the number, dimensions, and degrees of the irreducible components of fibers $\mathcal{Z}_h$, for $h \in U$, are bounded.
Here, as in the statement, ``degree'' is taken with reference to the reduced scheme structure. 
 
The number of irreducible components of any geometric fiber of $\mathcal{Z}$ is bounded by \cite[9.7.9]{EGAIV3}.
This bounds the number of geometric components, and so also the number of irreducible components, 
of any  $Z$ as in \eqref{Zdef}.  

Now we turn to the degree.  
For each $j$ with $0 \leq j \leq n$, 
let $Z_j$ be the closure in $\mathcal{Z}$ of the union of $j$-dimensional components of $\mathcal{Z}_{\eta}$ (thought of, for now, merely as a closed subset of $\mathcal{Z}$ in the Zariski topology; we will revisit the issue of scheme structure shortly.)   
In particular, the fiber over $\eta$ of each $Z_j$ is equidimensional of dimension $j$.
By \cite[9.5.1, 9.5.5]{EGAIV3}, we can restrict to an open $U \subseteq H$, on which:
\begin{itemize}
\item[(i)] For each $s \in U$, the fiber $(Z_j)_s$ is equidimensional of dimension $j$,
\item[(ii)] For each $s \in U$, the fiber $\mathcal{Z}_s$ is set-theoretically covered by the various $(Z_j)_s$, and
\item[(iii)] For each $s \in U$, and for all $j' < j$, the intersection $(Z_j)_s \cap (Z_{j'})_s$ has all components of dimension strictly less than $j'$.
\end{itemize}

Take $s \in U$ and let $K$ be any irreducible component of $\mathcal{Z}_s$, say of dimension $q$;  
by (ii) it is contained in some irreducible component of some $(Z_j)_s$.
Since $K$ is maximal among irreducible subsets, we must have equality here, i.e.\ 
$K$ coincides with an irreducible component of this $(Z_j)_s$, and
  since $(Z_j)_s$ is equidimensional of dimension $j$
we must have $j=q$.

Now let us endow $Z_j$ (so far merely a closed set) with its reduced scheme structure. 
 But now by generic flatness (\cite[6.9.1]{EGAIV2}), 
we further shrink $U$ so that each $Z_j$ is flat over $U$.
Then for each $j$, the degree of each $(Z_j)_s$ is independent of $s$.
Now, the scheme structure on $(Z_j)_s$ need not be reduced,  but
nonetheless 
there is an inequality of degrees $\mathrm{deg} \ (Z_j)_s \geq \mathrm{deg} \ (Z_j)_s^{\mathrm{red}}$
(where $(Z_j)_s^{\mathrm{red}}$ denotes the fiber taken with the reduced scheme structure)
and the degree of $(Z_q)_s^{\mathrm{red}}$   bounds from above
the degree of $K$ taken with its reduced scheme structure. 
 
Now we turn to (c).
Consider the geometrically irreducible components $W_1, \dots, W_h$ inside the base change $V_{\overline{F}}$  of $V$ 
to an algebraic closure.  
Note that $h$ is bounded by the degree of $V$, which we have assumed bounded,  
and similarly the degree of each $W_i$ is bounded by the degree of $V$.  
The intersection $\cap_{i=1}^h W_i$ (with its reduced structure)
is a Galois-stable closed subscheme of $V_{\overline{F}}$ and thereby descends to a reduced $F$-subscheme $W \subset V$.
Since $V$ is not geometrically irreducible, we know that $\operatorname{dim} W \leq \operatorname{dim} V - 1$.
Also (b), applied with $F$ replaced by $\overline{F}$, implies that $W_{\bar{F}}$ has a bounded number of irreducible components, each of bounded degree.
The same is then true for the $F$-scheme $W$. 
We claim, further, that $W(F) = V(F)$.
To see this, note that any $F$-point of $V$ must be Galois-invariant; 
since Galois permutes the geometric components of $V$ transitively,  the associated element of $V(\overline{F})$
belongs to all $W_i$. 
(See \cite[Lemma 0G69]{Stacks} for a similar argument.) 

The argument for (d) is similar to that for (b). Again we can parameterize all such $V$ by a suitable finite type Hilbert scheme $\mathrm{Hilb}$ 
and, writing $\mathcal{H} \rightarrow \mathrm{Hilb}$ for the universal subscheme, the smooth locus of $\pi$ coincides
with the set of points of $\mathcal{H}$ that are smooth in their fiber over $\mathrm{Hilb}$
(see e.g.\ \cite[17.5.1]{EGAIV4}).  
 Let $\mathcal{Z}$
be the complement of this smooth locus, endowed with the reduced structure. Then proceed as in (b).  
\qed

Finally, the following theorem of Broberg \cite{Broberg}
builds on fundamental ideas of Heath-Brown \cite{HBR} and Bombieri-Pila \cite{BP}: 

\begin{thm}[Broberg, 2004]   \label{Brobs}
Let $V \subset \mathbb{P}^M_{K}$ be an irreducible closed subvariety of dimension $n$ and degree $d$.
Then the points
of $V(K)$ whose  naive height is at most $H$ are
contained in a set of $K$-rational  divisors of  cardinality $\ll_{\eps, M} H^{\frac{n+1 +\eps}{d^{1/n}}  }$,
and each of which has degree $O_{\eps, M}(1)$.
\end{thm}

We note that, as stated in \cite{Broberg}, Broberg requires a bound on the generation of the ideal of $V$. This bound is however automatic from Lemma \ref{Kleiman} part (a). 
 Also ``divisor'' in the statement means ``effective Cartier divisor.''  

\begin{rem} \label{Brobremark}
Because Theorem \ref{Brobs} is so crucial, particularly its dependence on degree, 
we briefly outline 
where it comes from, taking $K=\Q$ to simplify notation; this is not needed in the remainder of the paper. 

One chooses a large integer $k$ 
and embeds $V \hookrightarrow \mathbb{P}^{e-1}$
via  a basis of sections of $\Gamma(V, \mathcal{O}(1)^{\otimes k})$.  In fact, we can and do choose from $\Gamma(\P^M,  \mathcal{O}(1)^{\otimes k})$ a set of monomials $f_1, \ldots, f_e$ of degree $k$ in the $M+1$-variables which freely span $\Gamma(V, \mathcal{O}(1)^{\otimes k})$.

Choose a ``good'' prime $p$ and examine a collection of $e$ points $P_i \in V(\Q)$
which all reduce to the same point modulo $p$ of $\P^M$, and whose height is at most $H$.  Expressing each $P_i$ in coprime integer coordinates, 
we can speak of the evaluation $f_i(P_j) \in \mathbb{\Z}$.    

Consider
$$ \Delta := \det \left[ f_i(P_j) \right]_{1 \leq i,j \leq e} \in \mathbb{Z}$$
which measures the volume of the $e$-simplex in $\mathbb{Z}^e$ spanned by the $P_i$ and the origin. 
On the one hand, 
 $\Delta$ is bounded by a constant multiple of $H^{ke}$. 
 
On the other hand, $\Delta$ is highly divisible by $p$,
because the values of $f_j$ modulo power of $p$ are highly constrained,
and therefore there are many relations (mod $p^k$) between rows of $\Delta$. 
To see this more formally, fix $r \geq 1$ and let 
denote by $V(\Z_p)_0$ the subset of $V(\Z_p)$
consisting of points with a given reduction modulo $p$. Set
$$M_r := \{ \mbox{functions: $V(\Z_p)_0 \rightarrow \Z/p^r$}\}.$$
Each $f_j$ gives an element of $M_r$, by evaluation and reduction modulo $p^r$. 
For $r=1$,  all these functions (for varying $j$) lie in a $\Z/p$-module of rank one: the constant functions.
For $r=2$, these functions
depend only on the  ``constant term'' and ``derivative'' of $f_j$,
and thereby 
lie in a  $\mathbb{Z}/p^2$-submodule of $M_2$  of rank $n+1$.
For $r=3$ we get $\mathbb{Z}/p^3$-submodule of $M_3$ of rank $(n+1)(n+2)/2$,
where the number comes from counting  possible 
Taylor expansions of $f_j$ up to degree two.  
Each such statement gives linear constraints on the rows of $f_i(P_j)$, 
and therefore leads to divisibility for $\Delta$. 
    
Computing with this we find 
$$v_p(\Delta) \gtrsim k e \cdot \frac{ d^{1/n} }{1+1/n},$$
where $d$ arises on the right-hand side 
eventually through the asymptotic behaviour of $e = \dim \Gamma(V, \mathcal{O}(k))$,
cf.\ \eqref{growth}.
Choosing  $p$ so that
$p^{v_p(\Delta)}$ is larger than the size bound $H^{ke}$ then forces $\Delta=0$; 
so the points $P_i$ lie on a hyperplane of $\mathbb{P}^{e-1}$, i.e.
all the points of $V$ with a fixed mod $p$ reduction lie on a divisor. 
So we produce $\approx p^{n}$
divisors covering the points of height $\leq H$. The argument above is so flexible -- in particular, using freedom to choose $p$ --
that one can achieve bounds that are uniform in $V$.

Notice the crucial point: as the degree of $\mathcal{O}(1)$  on $V$ increases, $\mathcal{O}(1)^{\otimes k}$ has more sections for a fixed $k$, giving
stronger divisibility for $\Delta$ and thus stronger bounds.

\end{rem}

\section{Reduction to Theorem \ref{thm2}}
We describe how Corollary \ref{Corhyp}  is reduced to  Theorem \ref{thm2}. 
\label{Main1}

\subsection*{Deduction of Corollary \ref{Corhyp} from Theorem \ref{thm1}} \label{Corproof}

We will focus on the first  statement of the Corollary, with $S$-integral points, 
and remark at the end of the proof on the only modification needed to handle
the statement about fixed discriminant.

For $n=2$ (binary forms of degree three and above) one in fact knows finiteness (Birch--Merriman \cite{BM}). 
The same is true for the case $n=4, d=3$ of cubic surfaces (Scholl \cite{scholl:delpezzo}). 
 
We may now restrict to the remaining cases $n \geq 3, d \geq 3$ and $(n,d) \neq (4,3)$.
 We will apply  Theorem \ref{thm2} taking $X=\mathbb{P}^M$ the projective space
 parameterizing polynomials of degree $d$ in $n$ variables up to scaling,
 with $Z$ the zero-locus of the discriminant, and taking the geometric variation of Hodge structure to arise
 from the middle cohomology of the universal family of hypersurfaces over $X^{\circ}$.  
Here the infinitesimal Torelli theorem is known, see \cite[Thm.\ 9.8(b)]{Griffiths}; that is, every fiber of the period map is locally contained in an orbit of $\mathrm{PGL}_n(\C)$; so,
noting that the Weil height of $[a_0: \dots : a_M]$ is given by $\prod_{v} \max_{i} |a_{i,v}|$ and is thereby bounded  
by a power of $B$ in the situation of the Corollary, 
Theorem \ref{thm2} shows that the integral points in question are covered by $O_{\eps}(B^{\eps})$ orbits of $\PGL_n(\C)$
or equivalently $\GL_n(\C)$.

We must replace $\GL_n(\C)$ by $\GL_n(\Z[S^{-1}])$. 
This is not difficult
but one must take care because a hypersurface could have automorphisms in characteristic $p$ that do not lift to characteristic zero.

Fix $P_0$ as in the Corollary. We  will show that the number of $\GL_n(\Z[S^{-1}])$-orbits on
$S$-integral polynomials  $P \in \GL_n(\C)  P_0$  with $S$-integral discriminant is bounded in terms of $n,d,S$.   Let $h$ be the degree of the discriminant polynomial.  For any  $S$-integral $P \in \GL_n(\C) P_0$ with $S$-integral discriminant, there exists a rescaling of $P$ by an $S$-unit
 whose discriminant  has $p$-valuation between $0$  and $h$. 
 It suffices, then, to show that for any integer $N = \prod_{p \in S} p^{a_p}$ (with $0\leq a_p < h$)
 the set of $S$-integral polynomials in $\GL_n(\C) P_0$ with discriminant $N$
lie in a union of $O_{n,d,S}(1)$  orbits of $\SL_n(\Z[S^{-1}])$. 

Now, write  $Y$ for the affine hypersurface defined by $\mathrm{disc}(P)=N$,
which we can regard as an affine scheme over $\Q$ (and even over $\Z$).
It is equipped
with an action of the $\Q$-algebraic group $G=\mathrm{SL}_n$.
We must show that the intersection of $Y(\Z[S^{-1}])$ with any $G(\C)$-orbit
is covered by $O_{n,d,S}(1)$ orbits of $G(\Z[S^{-1}])$.  We will need:

\begin{quote}  {\em Claim 1:} The action morphism
$G \times Y \rightarrow Y \times Y$
is a {\em finite} morphism of algebraic varieties over $\Q$. 
\end{quote}

This  follows essentially from the theorem of Matsumura and Monsky \cite{MM} that each stabilizer 
$G_y$ for $y \in Y(\C)$ is finite.
The deduction can be carried out using results of geometric invariant theory
to show that the stack of smooth hypersurfaces is separated; see
\cite{JLmoduli}.
We give a self-contained argument using similar ideas.

\proof (of {\em Claim 1.}) 
It is sufficient
to prove
this over $\C$, and, since the morphism is quasi-finite, it is sufficient to prove that it is proper. 

Using the singular value decomposition
one reduces to checking that the action of the diagonal subgroup $T = \{ (t_1, \dots, t_n) \text{ with }
t_i \in \R_+ \}$ on $Y(\C)$ is proper, which can be checked for the analytic topology \cite[XII, Prop 3.2]{SGA1}.  In other words, one must show that for any compact regions $\Omega_1$ and $\Omega_2$ in $Y(\C)$, the set $\{g \in T: g\Omega_1 \cap \Omega_2 \mbox{ nonempty} \}$ is bounded.  It suffices to show that, if  $P = \sum a_{i_i \dots i_n} x_1^{i_1} \dots x^{i_n}$ and $Q = \sum b_{i_1\dotsi_n} x_1^{i_1} \dots x^{i_n}$ and $(t_1, \ldots, t_n) \cdot P = Q$ then the absolute values of the $t_j$ can be bounded in terms of the coefficients of $P$ and $Q$.

Write $\Sigma$ for the set of $I = (i_1, \ldots, i_n)$ such that $a_I \neq 0$.  For each $t = (t_1, \ldots, t_n)$, write $t^I$ for $\prod_j t_j^{i_j} = \exp(\sum i_j \log t_j)$.  Then the maximal absolute value of a coefficient of $(t_1, \ldots, t_n) \cdot P$ is $\max_{I \in \Sigma} |a_I| t^I$, so the condition that $(t_1, \ldots, t_n) \cdot P = Q$ provides upper bounds on $|a_I| t^I$ for all  $I \in \Sigma$. 
 
An upper bound on $|a_I| t^I$ restricts $(\log t_1, \ldots \log t_n)$ to a half-space; we are done once we show that the intersection of these half-spaces over all $I \in \Sigma$ is a compact region in $T$.  
This is the case exactly when the region 
\begin{equation} \label{badset} \{ t \in T: \sum i_j \log t_j > 0 \mbox{ for all $I \in \Sigma$}\}\end{equation}
is {\em empty}. Suppose otherwise; then there exists $t$ in this region 
of the form $t_0^{m_1}, \ldots, t_0^{m_n}$ for some (whence any) $t_0 \in \R_+$ and with $m_j \in \Z$. 
By assumption on $t$, the limit in the analytic topology $P_0 := \lim_{t_0 \rightarrow 0} (t_0^{m_1}, \ldots, t_0^{m_n}) \cdot P$ exists; explicitly,
$P_0 = \sum_{I \in \Sigma_0} a_I t^I$
where 
$\Sigma_0$ is that subset of $\Sigma$ consisting of those $I$ with $\sum i_j m_j = 0$.
Then $\disc(P_0) = \lim_{t_0 \ra 0} \disc (t \cdot P) = \disc(P)$ is nonzero, so $P_0$ cuts out a smooth affine hypersurface.   
But the identity $
\sum m_j X_j \frac{d}{dX_j} P_0 = 0$
means that the $n$ conditions $\frac{dP}{dX_j} = 0$ cutting out the nonsmooth locus  are dependent, 
contradicting smoothness of $P_0=0$
(cf. proof in \cite{MM}).
Thus region \eqref{badset} is empty, so the action of $T$ is proper, so the action of $G$ is proper as well.
This concludes the proof of {\em Claim 1.} \qed
 
Take $y_1, y_2 \in Y(\Z[S^{-1}])$. We will now prove
\begin{quote}
{\em Claim 2 :} The action of $\Gal(\overline{\Q}/\Q)$
on the stabilizers $G_{y_i}(\overline{\Q})$  and also on the set $G_{12} := \{g \in G(\overline{\Q}): g y_1 = y_2\}$
is unramified outside a set of primes $\mathcal{P}$ depending only on $n,d,S$. 
\end{quote} 
 
\proof (of {\em Claim 2}). 
It is enough to prove the claim for $G_{12}$; take $y_1=y_2$ to get the claim about stabilizers. 
 
From finiteness of the action map we see that the matrix entries of
$g^{\pm 1}$ for $g \in G_{12}$ satisfy monic polynomials whose coefficients are rational polynomials in the coordinates of $y_i$.   Take $A_1(y_1,y_2), \ldots, A_M(y_1,y_2)$ to be the finite collection of all coefficients arising in this way.

Now,  for any extension of the $p$-valuation on $\Q$ to $\overline{\Q}$, and for every matrix entry $g^{\pm 1}_{ij}$ of $g^{\pm 1}$  we have 
\begin{equation} \label{cb}  |g^{\pm 1}_{ij}|_p  \leq \max_k |A_k(y_1, y_2)|_p. \end{equation}
Take $P$, larger than any prime in $S$,  such   that the coefficients of each $A_i$ are $p$-integral for all $p > P$; it now follows from \eqref{cb} that, for all $p > P$, any element $g$ which lies in $G_{12}$ for any $y_1,y_2 \in Y(\Z[S^{-1}])$ has for entries roots of monic polynomial with $p$-integral coefficients; in other words, it is  $p$-integral.
Thus, for all such $p$ there is an induced map
$$ G_{12}  \rightarrow \SL_n(\overline{\Z}_p) \rightarrow \SL_n(\overline{\mathbb{F}}_p)$$
which is necessarily injective if we also require $p$ to be larger
than the order of $G_{12}$, since  any torsion element of the kernel of  $\SL_n(\overline{\Z}_p) \rightarrow \SL_n(\overline{\mathbb{F}}_p)$
must have $p$-power order. For general constructibility reasons (similar to Lemma \ref{Kleiman}) 
the order of $G_{12}$ is bounded above; choose $P$ to be larger than this bound.  

Now for any $\sigma$ belonging to the inertia group at $p >P$, and any $g \in G_{12}$,  
$g^{\sigma}$ and $g$ have the same image in $\SL_n(\overline{\mathbb{F}_p})$, so must coincide.  This is precisely to say that the Galois action on $G_{12}$ is unramified at $p$, so we have proved  {\em Claim 2}
with $\mathcal{P}$ the set of primes less than $P$. 
\qed

 Now, with $y \in Y(\Z[S^{-1}])$,  there  is an injection
 $$ \mbox{$G(\Q)$-orbits on $Gy \cap Y(\Q)$} \hookrightarrow \mbox{$G_{y}$-torsors over $\Q$}$$
 sending $y' \in Y(\Q)$ to the right torsor $\{g: gy=y'\}$.  The {\em Claim} means that,
 under this map, elements of $Y(\Z[S^{-1}])$ are sent to  unramified-away-from-$\mathcal{P}$
 torsors for the unramified-away-from-$\mathcal{P}$-group $G_y$, whose order is moreover bounded.  Hermite-Minkowski gives an upper bound on the size of unramified Galois $H^1$ in this setting, and we conclude that
 the {\em  $S$-integral} points lying in $Gy \cap Y(\Q)$ lie in a collection of $G(\Q)$-orbits
 whose cardinality is bounded in terms of $d,n,S$. 
 Finally, we pass to $G(\Z[S^{-1}])$-orbits using \eqref{cb}.  
 This proves the first statement of the Corollary.

 In the last sentence of the argument we used \eqref{cb} only at $p \notin S$;
 but using it also for $p \in S$ allows one to conclude similarly that
 \begin{equation} \label{wwp2} \{y \in Y(\mathbb{Z}): \mathrm{disc}(y) =N,  \mbox{ht.}(y) \leq B\}\end{equation}
 is covered by $O_{\eps}(B^{\eps})$ orbits 
 of $\SL_n(\Z)$.  Here $\mathrm{ht}(y)$ just refers to the largest coefficient of $y$, rather than a Weil height. This is the second statement of the Corollary.

\subsection*{Reduction of Theorem \ref{thm1} to Theorem \ref{thm2}}

 Suppose that $X^{\circ}$ is quasi-projective and  $X^{\circ}_{\C}$
 admits a geometric variation of Hodge structure, i.e.\ there is a morphism
 $$ \pi_{\C}:   \mathfrak{X}_{\C} \rightarrow X_{\C}^{\circ}$$
such that the variation of Hodge structure   
in the theorem statement is a direct summand of $R^i \pi_{\C*} \C$,
for some $i \geq 0$.  (That such $\mathfrak{X}$ exists is understood to be the content of the word ``geometric.'') 

  Theorem \ref{thm1} assumes that the period morphism associated to $\mathsf{V}$
is  locally finite-to-one.
In this case, the period morphism associated to $R^i \pi_{\C*} \C$ has
the same property.   
The issue
to be handled is that $\mathfrak{X}_{\C}$ is defined only over $\C$
whereas Theorem \ref{thm2} requires a $K$-morphism as input. 

To verify sparsity it will suffice,   by descending
Noetherian induction and extending the base field, to produce a proper 
Zariski-closed subset $E \subset X_{\overline{K}}$ 
with the property that integral points on $X^{\circ}$ that 
do not lie inside $E$ are sparse.  (Warning: this is not the same as studying
``integral points on the quasi-projective variety $X^{\circ} \backslash E$.'')

By a standard
spreading out technique we may extend 
$\pi_{\C}: \mathfrak{X}_{\C} \rightarrow X_{\C}$
over the spectrum of a subring $R \supset K$ of $\C$, finitely generated over $K$: 
$$ \pi: \mathfrak{X}_R \rightarrow X_R =  X \times_{\mathrm{Spec} \ K} S,$$
where $S$ is the spectrum of $R$. This recovers $\pi_{\C}$
upon taking the pullback via the map 
$s : \Spec \C \rightarrow S$
associated to $R \rightarrow \C$. 

The image of 
$s$ is the generic point  $\eta_S$ of $S$. 
Since $S$ is the spectrum of the finitely generated integral domain $R$, 
$S \rightarrow \operatorname{Spec} K$ is smooth at $\eta_S$. 
Then, deleting the nonsmooth locus, we may suppose that $S$ is 
a smooth $K$-variety.   

Consider the morphism of vector bundles  supplied by Lemma \ref{KS}, which, applied to the morphism $\mathfrak{X}_R \rightarrow X \times S$, gives
 a morphism of locally free sheaves  \begin{equation}  \label{g2} g: T_{X \times S} \otimes \mathcal{H}_1 \rightarrow \mathcal{H}_2\end{equation}
   over  $X \times S$. 
   Let $T_{X} \subset T_{X \times S}$ be the sub-bundle defined  by the pullback of the tangent bundle of $T_X$. 
   Since $\pi_{\C}$ has finite-to-one period map,     the specialization $g_{x,s}|_{T_X}$ is injective for each $x \in X(\C)$;   
   the same is then true for a Zariski-open neighbourhood
   $U$ of $X \times \{\eta_S\}$ inside $X \times S$. 
   
Choose $s' \in S(\overline{K})$
such that $X(\overline{K}) \times s'$ meets $U(\overline{K})$.
The fiber of $\mathfrak{X}_R \rightarrow X_R$ over
$s'$ gives a  proper smooth morphism of $\bar{K}$-varieties 
$\mathfrak{X}_{s'} \rightarrow X_{\bar{K}}$.
The associated period map has generically injective derivative; let $E \subset X_{\bar{K}}$ be the locus where its
derivative has a nontrivial kernel, i.e., where $g_{x,s'}|T_X$ fails to be injective. 
$E$ is a proper Zariski closed subset of $X$ defined over some finite extension $K_1 \supset K$, and  
Theorem \ref{thm2} (applied after passage to $K_1$)  implies that integral points on $X^{\circ}$ that lie on the complement of $E$ are sparse.  
This establishes the inductive step, and therefore
concludes the reduction of Theorem \ref{thm1} to Theorem \ref{thm2}.

\section{Proof of Theorem \ref{thm2}}
The remainder of the paper is devoted to the proof of Theorem \ref{thm2}. We use notation as in the statement.
We will fix throughout a good integral model for both $(X,Z)$ and the morphism $\mathfrak{X} \rightarrow X^{\circ}$,
as in Definition \ref{goodmodel}.  

As we have noted, the proof involves proving that integral points of $X$ lie on
various collections of subvarieties, whose dimension will be steadily reduced until they
are either points or fibers of the period map.  
The key inductive statement used to reduce the dimension of the subvarieties is
Lemma \ref{lem:induction}.  

It may be helpful to note, in advance, that we will not need to keep track of any {\em integral} structure on our subvariety. The notion
of ``integral point on a subvariety'' will simply mean a $K$-rational point of the subvariety that is integral
as a rational point on $X$. 

\subsection{Large fundamental group.} \label{large}

Enlarging $S$ if necessary, we fix a prime 
 $p \in S$ for which the integral cohomology of complex  fibers of $\mathfrak{X}$ over $X^{\circ}$
is $p$-torsion-free, say of rank $r$ over $\Z$. 
Fixing $x \in X^{\circ}(\C)$ 
we get a monodromy representation of the topological fundamental group
\begin{equation} \label{Gmdef} \pi_1(X^{\circ, \an}, x) \longrightarrow G_n :=  \mathrm{Aut}(H^i(\mathfrak{X}_x, \Z/p^n)) \simeq \GL_r(\Z/p^n \Z).\end{equation}

In this setting, the {\em global invariant cycle theorem} (\cite[Corollaire 4.1.2 and 4.1.3.3]{Deligne_Hodge2}) implies the following statement:  For any complex irreducible subvariety $\iota: V \hookrightarrow X_{\C}$,
 not contained in $Z_{\C}$ and with $V^{\circ}$ not contained in a fiber of the period map,   
 the image of the monodromy representation 
\begin{equation} \label{i-inf}  \pi_1(V^{\circ,\an}) \rightarrow  G_n \end{equation}
(topological $\pi_1$, taken for an arbitrary choice of basepoint) has size that grows without bound as $n \rightarrow \infty$. 
Indeed \cite[Corollaire 4.1.2 and 4.1.3.3]{Deligne_Hodge2}
applies, after passing from $V$ to the smooth part $V'$ of its intersection
with $X^{\circ}$, to show that the image of  
$$ \pi_1({V'}^{\an}) \longrightarrow \mathrm{Aut} \ H^i(\mathfrak{X}_x, \Q)$$
is infinite.  
In particular, the image of the monodromy representation of $\pi_1((V^{\circ})^{\an})$
on $H^i(\mathfrak{X}_x, \Z)$ is infinite, 
so the size of the image of monodromy on $H^i(\mathfrak{X}_x, \Z)/p^n$
grows without bound as $n \rightarrow \infty$.  

These results transpose, as usual, to the {\'e}tale topology.
Indeed,   $R^i \pi^{\et}_*(\Z/p^n \Z)$  defines
a locally constant {\'e}tale sheaf  of $\Z/p^n$-modules on $X^{\circ}$,
which, by the local constancy of direct images for a smooth proper morphism (\cite[Theorem 5.3.1]{SGA4.5}),
extends to a locally constant {\'e}tale sheaf on the good integral model $X^{\circ}_S$ over $\mathfrak{o}_S$. 
  
For $V$ as above, standard comparison theorems show that the homomorphism in \eqref{i-inf} factors as
\[ \pi_1(V^{\circ,\an}) \rightarrow \pi_1^{\et}(V^{\circ}) \rightarrow G_n, \]
so
 the image of {\'e}tale $\pi_1$ 
\begin{equation} \label{fff} \pi_1^{\et}(V^{\circ}) \rightarrow G_n\end{equation}
(again, with an arbitrary geometric basepoint in $(V^{\circ})$)
has size that grows without bound as $n \rightarrow \infty$.

Next, let $E \supset K$ 
be an arbitrary algebraically closed field, with base change $X_E := X \times_{K} E$;
let $i: V \hookrightarrow X_E$ be a closed $E$-subvariety
not contained in $Z_E$ or a fiber of the period map.
(Note that we can make sense of the latter condition without reference to $\C$ by using Lemma \ref{KS}: 
by ``$V$ is contained in a fiber of the period map'' we mean that the associated morphism of vector bundles is zero on the smooth locus of $V$.)

We get by base change $\pi_E: \mathfrak{X}_E \rightarrow X_E$ and an {\'e}tale local 
system  $R^i \pi^{\et}_{E*} (\Z/p^n \Z)$  on $X_E^{\circ}$; and  
the same conclusion as above holds, i.e.\
the monodromy representation \eqref{fff} for $V$ on $R^i \pi^{\et}_* (\Z/p^n \Z)$ has ``large image''
in the sense specified above. 
We will use this only in the case when $E$ is the algebraic closure
of a finitely generated field; we may then choose a $K$-embedding $\sigma: E \rightarrow \C$,
and thus also $V_{\C} \subset X_{\C}$ compatibly with $V \subset X_E$. 
The local system $R^i \pi^{\et}_{E*} (\Z/p^n \Z)$ on $V_E$
pulls back to the similarly defined system on $V_{\C}$,
so ``large monodromy'' for $V$ follows from the same statement for $V_{\C}$. 
 
\subsection*{Construction of a suitable cover of $X$} \label{bigcover}

Our proof will involve an induction over higher and higher-codimension subvarieties of $X$ about which we know almost nothing apart from their degree.
It is thus crucial to have at hand covers of $X$ whose monodromy is uniformly bounded below on restriction to {\em every} subvariety of bounded dimension and degree.
  
\begin{lem}\label{lem31}  Fix $d, n, D \geq 1$,
and let $H$ be any (locally closed, finite type) complex subvariety of the Hilbert scheme of  
subschemes of $X_{\C}$ of degree $\leq d$ and dimension $n$.

There are a finite group $G$ and  a 
finite morphism $f: \widetilde{X} \ra X$ of $K$-varieties, 
equipped with an injection $G \hookrightarrow \mathrm{Aut}(\widetilde{X} / X)$, \ 
such that:
\begin{itemize}
 \item[(a)] $f|_{X^{\circ}}$ is finite {\'e}tale Galois with deck group $G$, 
 and, moreover, extends to a finite {\'e}tale cover of the good integral model $X^{\circ}_S$. 
 \item[(b)]
Let $U \subset X$ be any $n$-dimensional irreducible  closed complex subvariety  
of degree $\leq d$.
Suppose that:
\begin{itemize}
\item The point of the Hilbert scheme classifying $U$ lies in $H(\C)$, and
\item $U$ is not contained in $Z$ and $U^{\circ \an}$ is not contained in a single fiber of the period map $\Phi$.
\end{itemize}

Let $Q$ be any irreducible component of $f^{-1} U$, endowed with the reduced structure,
and such that the induced finite map $f:Q \rightarrow U$ is dominant
(note that it is automatically \'etale over $U^{\circ}$).  
Then the degree of $f:Q \rightarrow U$ at the generic point is $\geq D$,
i.e., the induced map of function fields has degree $\geq D$. 

\end{itemize} 
\label{lem:bigcover}
\end{lem}

\begin{proof}

Through the rest of this proof, $U$ will represent a single subvariety of $X$, classified by a point of the Hilbert scheme,  
and we will use $\mathcal{U}$ for the universal family. 

 We are going to find a cover $f: \widetilde{X} \ra X$
 and a proper Zariski-closed subset $H_1 \subseteq H$
 such that the conclusion of (b) holds for any $U=U_h$,
 satisfying the assumptions of (b), and 
 with $h \in (H-H_1)(\C)$. The result will follow by Noetherian induction.
 In particular, removing the singular locus of $H$ at the start,
 we may suppose that $H$ is smooth.   

Take a geometric generic point $\eta \rightarrow H$.  
Let $U_{\eta} \subset X_{\eta}$ be the 
corresponding generic subscheme. 
We may assume without loss of generality that: 
\begin{quote}
{\em Situation:} 
\begin{itemize}
\item[(a)] Every geometric fiber of $\mathcal{U} \rightarrow H$ is integral;
\item[(b)]  Every fiber of $\mathcal{U} \rightarrow H$ meets $X^{\circ}$; 
\item[(c)] On each  fiber $\mathcal{U}_h$
for $h \in H(\C)$
the period map is not locally constant. 
\end{itemize}
\end{quote}

For (a), note that
the locus of points with geometrically integral fiber by \cite[12.2.1(x)]{EGAIV3}
is open on the base, so if there exists one $h \in H(\C)$
for which the fiber $U_h$ is integral, then 
(after shrinking $H$ to a suitable nonempty open neighbourhood) we can suppose it is true for all $h$. If there is no such $h$, then the conclusion of the theorem holds for $H$ vacuously.

 For (b), note that the set of $h$ for which $U_h$ meets $X^{\circ}$ is constructible.
To see this, first note that $U_h$ is reduced for every $h$.  
Now note that $U_h$ meets $X^{\circ}$
if and only if $(U_h \cap Z)_s \rightarrow U_s$ is not surjective, and apply \cite[9.6.1(i)]{EGAIV3}.
Thus, restricting to an open subset of $H$, 
we can assume that $U_h$ meets $X^{\circ}$ either for no $h$ or for all $h$.
In the former case, the statement is vacuously true;
so we can assume that (b) holds for all $h$.
  
For (c) let $\mathcal{U}'$ be the smooth locus of the morphism $\mathcal{U} \rightarrow H$, 
 which, by flatness of the morphism, coincides with the locus of points which are smooth points of their fibers (\cite[17.5.1]{EGAIV4}). 
 Note that:  
 \begin{itemize}
 \item $\mathcal{U}'$ is itself smooth over $\mathrm{Spec} \ K$, since it is smooth over $H$ and $H$ was assumed smooth. 
 \item  $\mathcal{U}'$ contains an open dense subset of every fiber, since these fibers are all integral.
 \item $\mathcal{U}'$ is a $K$-variety: this follows from the previous conditions.
 It is reduced by smoothness, 
and since $\mathcal{U}' \rightarrow H$ is flat (\cite[Tag 01VF]{Stacks}), $H$ is irreducible, and the fibers
are irreducible, it readily follows (\cite[Tag 004Z]{Stacks}) that $\mathcal{U}'$ is itself irreducible.
\end{itemize} 

The tangent bundle $T_{\mathcal{U}'}$ has a sub-bundle $T_{\mathcal{U}'/H}$
 made up of`vertical'' vector fields.  Restricting the morphism of Lemma \ref{KS}  to this sub-bundle we get
 $$ g: \mathcal{H}_1 \otimes T_{\mathcal{U}'/H} \rightarrow \mathcal{H}_2$$
 Now, we may certainly assume there is some $h \in H(\C)$
such that $\mathcal{U}_h$ satisfies the conditions of (b) in the statement of the Lemma, or else the Lemma once again holds vacuously.
In particular, there exists a  point $u \in \mathcal{U}_h(\C)$, smooth in the fiber $\mathcal{U}_h$, 
such that $g_u$ is nonzero. It follows that $g_u$ is nonzero on a nonempty Zariski-open
subset of $\mathcal{U}'$;  the image of this Zariski-open
by the dominant morphism $\mathcal{U}' \rightarrow H$ contains an nonempty open subset of $H$,
and we replace $H$ by this open to   obtain the second part of the {\em Situation.}

 So we proceed assuming ourselves to be in the {\em Situation} above. 
 We continue to write $\mathcal{U}'$ for the smooth locus of $\mathcal{U}/H$
 and $\mathcal{U}'^{\circ}$ for the preimage of $X^{\circ}$ in $\mathcal{U}'$.   Recall that our assumptions
 guarantee that $\mathcal{U}'^{\circ}$ is fiberwise dense in $\mathcal{U}'$.

By \S \ref{large} we can find
an $m$ for which the image of the {\em geometric} monodromy representation of $\pi_1(U_{\eta}^{'\circ})$
in $G_m$ has size at least $D$ (same notation as in \S \ref{large}). 
This choice of $m$ determines a finite \'etale Galois cover  of $X^{\circ}$ with Galois group $G_m$, which extends to a finite {\'e}tale Galois 
cover of the $\mathfrak{o}_S$-model $X^{\circ}_S$.

Let $f: \widetilde{X} \rightarrow X$ be the normalization of $X$ in this $G_m$ cover;
then $\widetilde{X}$ is a normal $K$-variety and the morphism $f$ is finite (although not necessarily flat). 
The action of $G_m$ by deck transformations above $X^{\circ}$ extends uniquely 
to a $G_m$-action on the morphism $f$.

The morphism $f$ gives of course a morphism $f: \widetilde{X}_{\eta} \rightarrow X_{\eta}$
after base-change from $\operatorname{Spec} \ K$ to $\eta$. The restricted map
\begin{equation} \label{fU0} f_{\eta}^{-1} \mathcal{U}_{\eta}'^{\circ} \rightarrow U_{\eta}'^{\circ} \end{equation} 
is finite {\'e}tale and  has degree $\geq D$ restricted to 
each geometric component of the source
by choice of $f$. 
Now this morphism is the geometric generic fiber of a
finite {\'e}tale morphism of smooth $H$-schemes: 
\begin{equation} \label{fU} f^{-1} \mathcal{U}'^{\circ} \rightarrow \mathcal{U}'^{\circ}\end{equation}
and we want to draw the same conclusion about degrees
for the fibers of \eqref{fU} over
a nonempty open subset of $H$. 
This will imply the desired conclusion, for -- with $Q$ as in the statement --
the assumed dominance implies that $Q \cap f^{-1} U'^{\circ}$
is an open nonempty subset of $Q$.

We now use \cite[Lemma 055A]{Stacks} 
(see also \cite[Prop.\ 9.7.8]{EGAIV3} and references therein). 
which guarantees the existence of a morphism $g: H' \rightarrow H$ (which in fact factors as a  finite {\'e}tale
  surjection followed by an open immersion)   such that, after base change of   
$f^{-1} \mathcal{U}'^{\circ} \rightarrow \mathcal{U}'^{\circ}$
 by $g$ -- i.e.\ replacing $\mathcal{U}'^{\circ}$ by $\mathcal{U}'^{\circ} \times_{H} H'$
 and similarly for $f^{-1} \mathcal{U}'^{\circ}$ -- 
  the following assertions hold:
\begin{itemize}
\item[(a)] Each irreducible   component of the generic fiber 
$(f^{-1} \mathcal{U}'^{\circ})_{\eta'}$ (with $\eta'$ the generic point of $H'$ -- not a geometric generic point here) is in fact
a geometrically  irreducible component of that generic fiber. 
\item[(b)] Let $\overline{Z_1}, \dots, \overline{Z_r}$ be the Zariski closures of these generic irreducible components $Z_1, \dots, Z_r$
inside $f^{-1} \mathcal{U}'^{\circ}$. 
These $\overline{Z_i}$ give, upon intersection with the fiber  $(f^{-1} \mathcal{U}'^{\circ})_{h'}$ above any $h' \in H'$, 
the decomposition of that fiber
into irreducible components, and indeed each of these irreducible components are geometrically irreducible.  \end{itemize}

In the decomposition of (b) 
 $$ (f^{-1} \mathcal{U}'^{\circ}) = \coprod \overline{Z_{\alpha}}$$
 the sets $Z_{\alpha}$ are disjoint.
 Indeed, upon restriction to each fiber, this decomposition recovers the decomposition of $f^{-1} \mathcal{U}'^{\circ}_{h'}$;
 however, this is finite {\'e}tale over $\mathcal{U}'^{\circ}_{h'}$ and Lemma \ref{disjointness} implies the disjointness.   
 In particular, the $\overline{Z_{\alpha}}$ are 
 both closed and open,  and in particular inherit a scheme structure as open sets in $f^{-1} \mathcal{U}'^{\circ}$. 
The restriction of the map $f$ to each $\overline{Z_{\alpha}}$ is  then a finite {\'e}tale map
 $f_{\alpha}:\overline{Z_{\alpha}} \rightarrow \mathcal{U}'^{\circ}$.
 The degree of such a map is locally constant on the base, and here, by assumption,
 that degree is $\geq D$ 
 everywhere on the generic fiber $\mathcal{U}'^{\circ}_{\eta'}$.   
 That generic fiber is dense because it contains every generic point of 
 of $\mathcal{U}'^{\circ}$, and, consequently   the degree of $f_{\alpha}$ is everywhere $\geq D$. 
Restricting to a single fiber $\mathcal{U}'^{\circ}_{h'}$   
for $h' \in H'$ gives the desired conclusion --
that is,  the  bound stated in (b) of the Theorem holds for all fibers $U_h$
for all $h$ in a nonenmpty open subset of $H$, explicitly, the image of $H' \rightarrow H$.
 
Finally, we conclude by Noetherian induction.
We have shown that, given $H$, there is a Zariski-closed $H_1 \subseteq H$
and an $m > 0$, such that the image of geometric monodromy in $G_m$ 
gives an \'etale cover for $X$, that satisfies the required properties for 
any $h \in (H - H_1)(\mathbb{C})$.
There is no harm in replacing $G_m$ by $G_{m'}$, for $m' \geq m$.
Thus we can apply Noetherian induction to find one $m$ that works for all $h \in H(\mathbb{C})$.
\qed

\subsection*{Bounding rational points on a subvariety} \label{ss:induction}
The following result is the key inductive step.
We note that {\em all constants appearing in this discussion are permitted to 
depend on the variety $X$, and indeed on the integral model 
chosen in \S \ref{integralmodel}, without explicit mention.}  
\begin{lem}
\label{lem:induction}
Let $V$ be a geometrically irreducible closed subvariety of $X$ defined over $K$,   
of dimension $n$ and degree $d$, such that $V_{\C}$
is not contained in $Z$ and $V^{\circ \an}$ is not 
contained in a fiber of $\Phi$. 

Then all integral points of $X^\circ$ of height $\leq B$
that lie on $V$ can be covered by 
$O_{d,\eps} (B^{\eps})$ irreducible (but not necessarily geometrically irreducible) subvarieties, all defined over $K$, 
with dimension $\leq n-1$ and degree $O_{d, \eps}(1)$.  \end{lem}
\proof 

Choose
$D$ so that $\frac{n+1}{D^{1/n}} < \eps$.
Let $\mathcal{P}$ be the
(finite, by
Lemma \ref{Kleiman}) set of Hilbert polynomials that arise   
from irreducible subvarieties of $X$ of dimension $n$ and degree $d$, 
and let $H_{\mathcal{P}}$ be the associated Hilbert scheme.  Write $H^{red}_{\mathcal{P}}$ for the reduced induced closed subscheme of $H_{\mathcal{P}}$.  

For this choice of $n,d,D, H=H^{red}_{\mathcal{P}}$ 
 take a finite group $G$ and a finite morphism  $f: \widetilde{X} \rightarrow X$ as provided by Lemma~\ref{lem:bigcover}.  We note that $H_{\mathcal{P}}(\C) = H^{red}_{\mathcal{P}}(\C)$, so the passage to the reduced subscheme structure is irrelevant for the statements on complex subvarieties proved in Lemma~\ref{lem:bigcover}. 

For every $x \in X^{\circ}(K)$, the action of $G$ on the fiber of $\widetilde{X}$ over $x$  
defines a class in the Galois cohomology group $H^1(\Gal(\overline{K}/K), G)$.
Concretely, since the Galois action on $G$ is trivial,
$H^1(\Gal(\overline{K}/K), G)$ classifies
homomorphisms $\rho: \Gal(\overline{K}/K) \rightarrow G$ up to conjugacy.
Choosing a point $\widetilde{x}$ above $x$,
we define a homomorphism $\rho$ by the rule
 \begin{equation} \label{tww} \sigma(\widetilde{x}) = \rho(\sigma) \cdot \widetilde{x}\end{equation}
for $\sigma$ in the Galois group.
If $x$ is $S$-integral, this homomorphism
$\rho$ is in fact unramified outside $S$. 
Such a $\rho$ can also be used to twist $\widetilde{X} \rightarrow X$, namely,
one modifies the Galois action on $\widetilde{X}$ through $\rho$;
and then \eqref{tww} means precisely that $x$ will lift to a $K$-rational point on the
twist of $\widetilde{X}$ indexed by $\rho$. 
(See  $\S$ 4.5 and Thm.\ 8.4.1 of \cite{PoonenRP} for further discussion.) 

There are only finitely many homomorphisms $\Gal(\overline{K}/K) \rightarrow G$, unramified outside $S$; 
 call them $\rho_1, \rho_2, \dots, \rho_R$. This list does not depend on $B$. 
 Each such $\rho_j$
can be used to twist $f$ to a map $f_j:\widetilde{X}_j \rightarrow X$.
Our previous discussion now shows that any
integral point of $X^\circ$  lifts along some $f_j$ to a point of $\widetilde{X}_j(K)$. 

For a sufficiently large integer $e$ the pullback $(f_j^* \LL)^{\otimes e}$ is
very ample and defines, after fixing a basis of sections,  a projective embedding  $\widetilde{X}_j \hookrightarrow \mathbb{P}^{M_j}$.
Now the data of the diagram of $K$-varieties and line bundles 
\begin{equation} \label{diag} (X, \mathcal{L}^{\otimes e}) \stackrel{f_j}{\longleftarrow}(\widetilde{X}_j, (f_j^* \mathcal{L})^{\otimes e}) \hookrightarrow (\mathbb{P}^{M_j}, \mathcal{O}(1))\end{equation}
depends on various choices, 
but these choices can (and will) be made once and for all depending only on $d,\eps$.  
Then for $P \in \widetilde{X}_j(K)$ we get 
\begin{equation} \label{imp}
H_\LL(f_j(P))^e \asymp H_{f_j^* \LL}(P)^e  \asymp H_{\mathbb{P}^{M_j}}(P)   
\end{equation} 
where the symbol $\asymp$ means that 
the ratio is bounded above and below by constants that may depend on $f_j$. 
Since there are only finitely many  $f_j$, and their coefficients are bounded in terms of $d$ and $\epsilon$ (and, as always, $X$ and $S$) but don't depend on $B$, 
these constants depend only on $d$ and $\epsilon$. 
  
Therefore, we have shown that  the integral points of $X^\circ$ with height $\leq B$ 
belonging to $V$  all have the form $f_j(P)$, where $P$ is a $K$-rational point of $f_j^{-1}(V)$ with $H_{\mathbb{P}^{M_j}}(P) \leq c_{d,\epsilon}B^e$.
It will suffice to prove the conclusion for those $P$ for which $f_j(P)$ is a {\em smooth} point of $V$,
simply by including each irreducible component of the singular locus of $V$ in the list of subvarieties (see Lemma \ref{Kleiman} part (d) for the necessary bounds). 
 
Let $V' \subset V$ be the (open) smooth locus. 
Consider those geometric components $Q^{\circ} \subset 
 (f_j^{-1} V')^{\circ}$  that  have a $K$-rational point. 
  Because $ (f_j^{-1} V')^{\circ}$ is a finite {\'e}tale cover
of the geometrically irreducible smooth $K$-variety $V'^{\circ}$, its 
geometric components are pairwise disjoint (Lemma \ref{disjointness})
and permuted by the Galois group;  so 
any such $Q^{\circ}$  is defined over $K$ and the number of such $Q^{\circ}$ is bounded in number
by the size of the group $G$.  

The Zariski closure  $Q$ of any $Q^{\circ}$ is again geometrically irreducible and defined over $K$;
we understand it to be endowed with its reduced scheme structure. 
The map $f_j: \widetilde{X}_j \rightarrow X$ induces a compatible map $f_j:  Q \rightarrow V$, 
which is  dominant  since, by construction of $Q$, the image contains a nonempty open set of $V'^{\circ}$. 
Indeed,  $f_j: Q \rightarrow V$ is  {\'e}tale over $V^{\circ}$,
with degree between $D$ and $(\# G)$; the lower bound comes from (b) of Lemma \ref{lem31},
using  also the fact that $f_j$ is a twist of $f$. 

The degree of $V$ with respect to $\mathcal{L}^{\otimes e}$ is $d e^n$, and therefore, 
by Lemma \ref{degree}  
the degree of $Q$, considered as a closed subvariety of $\mathbb{P}^{M_i}$  via \eqref{diag}, satisfies   
$$ D d e^n \leq \mathrm{deg}  Q \leq (\# G) d e^n.$$
 
We apply Theorem \ref{Brobs} to
each $Q$ that arises in the above fashion, i.e.\
to the 
 Zariski closure of any irreducible geometric component of $(f_j^{-1} V')^{\circ}$ that has a $K$-point.
Theorem \ref{Brobs} and our choice of $D$ implies that the 
set of  rational points of $Q$  of height $\leq c B^e$ 
are supported on a set 
of proper  closed subvarieties of $Q$ of degree $O_{d,\eps}(1)$ with cardinality $\ll_{d,\epsilon} B^{2\epsilon}$.  These subvarieties are defined over $K$ and need not be geometrically irreducible.

For any such $Q$ and any such proper subvariety $Y \subset Q$,  
the scheme-theoretic image $f_j(Y)$ under the finite map $f_j$  is a proper subvariety $f_j(Y) \subset V$, in particular, of dimension $\leq n-1$.  
Moreover,   $f_j$ restricts to a finite map $Y \rightarrow f_j(Y)$. 
By Lemma \ref{degree}  the  $\LL$-degree of $f_j(Y)$
is no larger than the $f_j^* \LL$-degree of $Y$, in particular, $O_{d,\eps}(1)$. 
 
The number of maps $f_j$ depends only  on $d,\eps$, 
and the number of $Q$ arising is then at most
the number of $f_j$
 multiplied by the order of $G$, which is again $O_{d,\eps}(1)$. 
  Consequently,  
 the number of  $Y$ arising as in the prior paragraph
 is $O_{d,\epsilon}( B^{2\epsilon})$, concluding the proof (after the obvious
 scaling $\epsilon \leftarrow \epsilon/2$.)   
 \end{proof}

\subsection{Conclusion of the proof of Theorem   \ref{thm2}}
\label{conclude}
\begin{proof} Fix $\eps > 0$.
We use descending induction via  Lemma \ref{lem:induction}. The inductive statement is the following:

\begin{quote}
$(\star)_n$: For every $n$ with $0 \leq n \leq \dim X$, there exists an integer $d_n$ with the following property: for all $B > 0$, 
the $S$-integral points of $X^\circ$ are covered by a collection of 
\[ O_{\eps}(B^{(\dim X - n) \eps}) \]
irreducible subvarieties of $X$, all defined over $K$, each of which is either
\begin{itemize}
\item[--]$\textrm{(a)}_n$: a subvariety of dimension $\leq n$ and degree $\leq d_n$, or   
\item[--]$\textrm{(b)}$:  a geometrically irreducible subvariety that is contained in a single fiber of the period map.
\end{itemize}
\end{quote}

The base case is given by $n = \dim X$, in which case, of course, the single subvariety $X \subseteq X$ suffices.

The implication  $(\star)_n \implies (\star)_{n-1}$ follows from Lemma~\ref{lem:induction}:
Let $\mathcal{V}_n$ be the collection of $n$-dimensional varieties in the statement of $(\star)_n$.
For each $V \in \mathcal{V}_n$, we will
construct a set of varieties covering all the integral points of $X^{\circ}$ lying on $V$. We subdivide into cases:

\begin{itemize}
 
\item $V$ is not geometrically irreducible. 
In this case, 
we take the set $\{V_i\}$ of subvarieties given by 
part (c) of Lemma \ref{Kleiman}. 
These varieties number at most
$O_{n,d_n}(1)$ and they have 
dimension $\leq n-1$ and degree $O_{n,d_n}(1)$.

\item $V$ is geometrically irreducible but $V^{\circ \an}$ is contained in a fiber of $\Phi$:  then    
we take the singleton set $\{V\}$. 

\item $V$ is contained in $Z$; in this case
we can take the empty set $\emptyset$. 

\item $V$ is geometrically irreducible and not contained in $Z$, and $V^{\circ\an}$ is not contained in a fiber of $\Phi$; then we may apply Lemma \ref{lem:induction} to show
that integral points of height $\leq B$ on $V$ are covered by $O_{d_n,\eps}(B^{\eps})$ irreducible $K$-varieties of dimension $\leq n-1$ and degree $O_{d_n,n,\epsilon}(1)$.  
\end{itemize}

We take $d_{n-1}$ to be the largest of the implicit constants $O_{n,d_n}(1)$ and $O_{n, d_{n}, \epsilon}(1)$
appearing in the above proof. Then, to sum up, 
by $(\star)_n$ we know that the $S$-integral points of $X^\circ$ of height at most $B$ are covered by $O_{\eps}(B^{(\dim X - n) \eps})$ subvarieties $V$  
satisfying either $\textrm{(a)}_n$ or (b), 
and we know that for each of those $V$, the subset of those points lying on $V$ is covered by $O_{\eps}(B^{\eps})$ subvarieties 
satisfying either $\textrm{(a)}_{n-1}$ or (b); 
together, these facts yield $(\star)_{n-1}$.  

We emphasize that this is the point in the argument where the uniformity in Broberg's result is crucial.  We have no control of the heights of the varieties making up the collection $\mathcal{V}_n$, and indeed these heights will grow with $B$; but since the implicit constants in Lemma~\ref{lem:induction} depend only on $X$ and $\eps$, not on $V$, this lack of control does not present a problem.

The case $n = 0$ gives the Theorem.
\end{proof}

\section{Author affiliations}
Jordan S.\ Ellenberg, University of Wisconsin

Brian Lawrence, University of California, Los Angeles; \url{brianrl@math.ucla.edu}

Akshay Venkatesh, Institute for Advanced Study

\bibliographystyle{plain}
\bibliography{References}

\end{document}